\documentclass[11pt]{amsart}

\usepackage{geometry}
\swapnumbers
\newtheorem{theorem}{Theorem}[section]
\newtheorem{lemma}[theorem]{Lemma}
\newtheorem{proposition}[theorem]{Proposition}
\newtheorem{corollary}[theorem]{Corollary}
\theoremstyle{definition}

\newtheorem{definitions}[theorem]{Definitions}
\newtheorem{remark}[theorem]{Remark}
\newtheorem{example}[theorem]{Example}

\usepackage{calrsfs}
\usepackage[T1]{fontenc}
\usepackage{textcomp}
\usepackage{times}
\usepackage[scaled=0.92]{helvet}

\usepackage{amssymb}

\newcommand{\G}{\mathbf{G}}
\newcommand{\Q}{\mathbf{Q}}
\newcommand{\Z}{\mathbf{Z}}
\newcommand{\N}{\mathbf{N}}
\newcommand{\F}{\mathbf{F}}
\newcommand{\A}{\mathbf{A}}
\newcommand{\Mat}{\mathrm{Mat}}
\newcommand{\End}{\mathrm{End}}
\DeclareMathOperator{\diag}{diag}
\DeclareMathOperator{\GL}{GL}

\begin{document}

\date{\today\ (version 1.0)}
\title{Matrix divisibility sequences}
\author[G.\ Cornelissen]{Gunther Cornelissen}
\author[J.\ Reynolds]{Jonathan Reynolds}
\address{\normalfont{Mathematisch Instituut, Universiteit Utrecht, Postbus 80.010, 3508 TA Utrecht, Nederland}}
\email{g.cornelissen@uu.nl, jonathan.reynolds@gmx.com}
\thanks{The idea of this paper arose after a discussion between the first named author and Thanases Pheidas in the early years of the 21st century. It was revived again through discussions and computer calculations with Graham Everest and Shaun Stevens during a visit to Norwich some time later --- and forgotten again. We thank all these people, along with Tom Ward, for their stimulating interest. The second author is supported by a Marie Curie Intra European Fellowship (PIEF-GA-2009-235210)}
\begin{abstract}
\noindent We show that many existing divisibility sequences can be seen as sequences of determinants of \emph{matrix divisibility sequences}, which arise naturally as Jacobian matrices associated to groups of maps on affine spaces. 
\end{abstract} 

\maketitle

\section{Introduction}

The most famous divisibility sequence is probably the Fibonacci sequence $\{F_n\}_{n \geq 1}$: if $m$ divides $n$, then $F_m$ divides $F_n$. This property is shared by other linear recurrent sequences \cite{Poort}, such as any other Lucas sequence, and by higher degree recurrent sequences known as elliptic divisibility sequences \cite{recseq, Ward}. Recent years have witnessed a revived and increasing interest in such sequences \cite{MR2164113, MR2045409, MR2301226, IMSSS, IngrSilv}, alongside applications in cryptography \cite{Shipsey, StangeL} and undecidability \cite{CZ, MR2480276}. In the current paper, we argue that -- in a non-tautological way -- behind each of these divisibility sequences lies hidden a naturally defined divisibility sequence of matrices, such that the given divisibility sequences occurs as the determinant of the sequence of matrices. 

The plan for the paper is as follows: we shall first introduce the general notion of a matrix divisibility sequence indexed by a semigroup. Then we will see how a faithful representation of the semigroup by endomorphisms of an affine space gives rise to a matrix divisibility sequence, by considering the Jacobian matrices of the endomorphisms. We will show how most of the commonly known divisibility sequences (mentioned briefly above) arise as determinants of matrix divisibility sequences through interesting semigroups of endomorphisms of affine spaces, often associated to a representation of addition in an algebraic group. For example, Lucas sequences are associated to the $2 \times 2$ Borel group. We also construct the \emph{elliptic matrix divisibility sequence} that underlies the usual elliptic divisibility sequences, and prove that it has primitive right matrix divisor classes. 

\section{Matrix divisibility sequences} 

In this section, we introduce general matrix divisibility sequences over a ring $S$, indexed by a semigroup $\Gamma$, and we define primitive divisor classes of matrix divisibility sequences. 

\begin{definitions}
Let $S$ denote a commutative unital ring and $\Mat_d(S)$ the ring of $d\times d$ matrices over $S$. A (right) \emph{divisor class of a matrix} $M \in \Mat_d(S)$ is a coset $\GL_d(S) \cdot M$ of $M$ in the left quotient of $\Mat_d(S)$ by the invertible matrices $\GL_d(S)$ over $S$. A matrix $M$ is said to (right) \emph{divide} a matrix $N$ if there exists a matrix $Q$ such that $N=QM$. If $M$ (right) divides $N$, then any element of the divisor class of $M$ also right divides $N$. 
\end{definitions}

\begin{example} \label{B} An interesting special case of matrix divisibility is that of integer matrices (i.e., $S=\Z$). In this case, the right divisor classes of a matrix $M$ are in bijection with subgroups of the cokernel $\Z^d / M^\top \Z^n$ of left multiplication by the transpose $M^\top$ of $M$, cf.\ \cite{Bhowmik}. 
\end{example} 
 
We will only consider right division from now on, and hence frequently leave out ``right'' from the terminology. 
 
\begin{definitions} \label{defdiv} Let $(\Gamma,\cdot)$ denote a (not necessarily commutative) semigroup. A divisibility sequence of matrices over a commutative ring $S$, indexed by $\Gamma$, is a collection of matrices $$\{M_\alpha\}_{\alpha \in \Gamma}$$ in $\Mat_d(S)$, such that if $\alpha$ right divides $\beta$ in $\Gamma$, then $M_\alpha$ right divides $M_\beta$ in $\Mat_d(S)$. A \emph{primitive divisor class} of a term $M_\alpha$ of such a sequence is a right divisor class of $M_\alpha$ that is not a right divisor class of any $M_\beta$ for $\beta$ a right divisor of $\alpha$. 
\end{definitions}

If $\{M_\alpha\}_{\alpha \in \Gamma}$ is a matrix divisibility sequence, then $$\{\det(M_\alpha)\}_{\alpha \in \Gamma}$$ is a divisibility sequence consisting of elements from the ring $S$. This is obvious from the multiplicativity of the determinant. 

In general, divisibility of matrices is strictly stronger than divisibility of their determinants. For example, the matrices $\diag(1,2)$ and $\diag(2,1)$ are not right or left divisors of each other over the integers, although of course, their determinants are. Thus, it appears that the theory presented here is a strict superset of the existing one. Over a PID, divisibility of matrices is in general also stronger than divisibility of their individual elementary divisors in the Smith Normal Form. 

\section{Matrix divisibility sequences arising from endomorphisms}

We produce a natural source of matrix divisibility sequences, as Jacobian matrices of endomorphisms of affine space. 

\begin{definitions} As before, let $(\Gamma,\cdot)$ denote a semigroup, $S$ a commutative unital ring. Now let $$[\cdot] \colon \Gamma \hookrightarrow \End(\A_S^d) \colon \alpha \mapsto [\alpha] $$
denote a faithful representation of $\Gamma$ into the group (under composition) of endomorphisms of affine $d$-space $\A_S^d$ over $S$ (i.e., morphisms $\A_S^d \rightarrow \A_S^d$).  Let $x \in \A^d(S')$ be a point in some ring extension $S \rightarrow S'$. The \emph{matrix divisibility sequence associated to $(\Gamma,[\cdot])$} is the sequence of Jacobians $\{J_\alpha\}_{\alpha \in \Gamma}(x)$, with $J_\alpha$ and $d \times d$ matrix whose $(i,j)$-entry is given by $$ (J_\alpha)_{i,j}:= \partial ([\alpha](x))_i) / \partial x_j. $$ The associated \emph{determinantal divisibility sequence} is given by $$ \{ \det(J_\alpha)(x)\}_{\alpha \in \Gamma}. $$ 
\end{definitions}

\begin{example} \label{Gm}
A trivial example: set $\Gamma = (\Z_{\geq 0}, \cdot)$, and $[n] \colon \A^1_{\Z} \rightarrow \A^1_{\Z} \colon x \mapsto x^n.$ Then indeed $[mn]=[m] \circ [n]$, and the associated (matrix) divisibility sequence is $nx^{n-1}$. At $x=1$, this is just the divisibility sequence of integers $1,2,3,\dots$. 
\end{example}

The following facts are obvious, but they represent the basic idea in our definition: \emph{derivatives turn composition into multiplication}. 

\begin{proposition}
A matrix divisibility sequence associated to $(\Gamma,[\cdot])$ as before is indeed a matrix divisibility sequence: if $\alpha$ right divides  $\beta$ in $\Gamma$, then for any $x \in \A^d(S')$, the matrix $J_{\alpha}(x)$ right divides $J_\beta(x)$ in the semigroup of $d \times d$-matrices $\Mat_n(S)$, and $\det(J_\alpha(x))$ divides $\det(J_\beta(x))$ in $S$. 
\end{proposition}

\begin{proof}
Write $\beta = \gamma \cdot \alpha$ in $\Gamma$. Then $[\beta] = [\gamma] \circ [\alpha]$. The chain rule for the Jacobian matrix implies that for any $x \in \A^d(S')$, we have $$ J_\beta(x) = J_\gamma([\alpha]x) \cdot J_\alpha(x) $$ in $\Mat_d(S)$. One can then simply take determinants of this identity. 
\end{proof}

\begin{remark} We have included the case of a general semigroup $\Gamma$, instead of focussing on the (positive) integers as index set for the sequence, because some natural examples arise from elliptic curves with complex multiplication \cite{Streng}, and even noncommutative semigroups occur naturally from supersingular elliptic curves over infinite fields of positive characteristic. 
\end{remark} 

\begin{remark} \label{bundle} 
A more general case would arise when one replaces affine space $\A^d$ by an algebraic variety $X$. If $[\cdot] \colon \Gamma \hookrightarrow \End(X)$ is a representation, then one may consider the pullback of $[\alpha]$ to the tangent bundle
$$ d[\alpha] \colon TX \rightarrow TX ,$$ which then satisfies the chain rule $$d[\alpha\beta](x)=d[\alpha](\beta x) \circ d[\beta](x).$$ Instead of taking a determinant, one may construct the highest exterior power $$\det d[\alpha] \colon {\bigwedge}^d\, TX \rightarrow  {\bigwedge}^d\, TX$$ as automorphisms of the canonical bundle ${\bigwedge}^d\, TX$. In general, however, there is no canonical choice for compatible coordinates in tangent spaces at different points (as there is on affine space), so that this does not lead to a ``numerical'' divisibility sequence. Therefore, we will not consider this more general setting here. 
\end{remark}

\section{A construction of endomorphisms from algebraic groups} \label{groupG}

A natural context for endomorphism representations is one that arises from the endomorphisms of a linear algebraic group, as follows. Let $(\G,+)$ denote an affine algebraic group over a field $k$, and let $\Gamma \subseteq \End_k(\G)$ denote a finitely generated semisubgroup of the algebraic group endomorphisms of $\G$. Fix an affine embedding of $\G$  into $\A^d$. Choose generators $\gamma_1,\dots,\gamma_n$ for the group, and fix an algebraic formula $\langle \gamma_i \rangle$ for the action of the generators on the affine embedding, and fix an algebraic formula for the product and inverse in the group in the given embedding. Now define a representation $[\cdot] \colon \Gamma \rightarrow \End(\A^d)$ by $[\sum a_i \gamma_i](x_1,\dots,x_d):=\sum a_i \langle \gamma_i \rangle(x_1,\dots,x_d)$, where $\Sigma a_i$ is computed using the given formulas for $+$ and $-$ in the group.

\begin{example}
Example \ref{Gm} fits into this framework, if we consider $x \mapsto x^m$ as iterates of the multiplication map on the multiplicative group $\G_m$. A more interesting example is the following: 
\end{example}

\begin{example}[Borel group and Lucas sequences] \label{borel}
Consider the Borel group $\mathbf{B}$ of $2 \times 2$ matrices with the affine embedding
$$ \mathbf{B} \rightarrow \A^3 \colon \left( \begin{array}{cc} X & Y \\ 0 & Z \end{array} \right) \mapsto (X,Y,Z), $$
and the multiplication formula
$$ (X_1,Y_1,Z_1) \odot (X_2,Y_2,Z_2) := (X_1 X_2, X_1 Y_2 + Y_1 Z_2, Z_1 Z_2), $$
corresponding to the product of matrices, and a similar one for the inverse. 
Now, for $n \in \N =\Gamma$, consider the endomorphisms given by $$[n](X,Y,Z) = \underbrace{(X,Y,Z) \odot \cdots \odot(X,Y,Z)}_{n \mbox{ {\footnotesize times}}} = (X^n, Y \frac{X^n-Z^n}{X-Z}, Z^n).$$
The associated matrix divisibility sequences of Jacobians of $[n]$ is 
$$ J_n(X,Y,Z) = \left( \begin{array}{ccc}  nX^{n-1} & 0 & 0 \\ YP(X,Z)  & \frac{X^n-Z^n}{X-Z} & YP(Z,X) \\ 0 & 0 & nZ^{n-1}  \end{array} \right), $$
with $$P(X,Z) = \frac{nX^{n-1}(X-Z) - (X^n-Z^n)}{(X-Z)^2}, $$
and the associated determinant sequence is 
$$ \det(J_n)(X,Y,Z)=n^2 X^{n-1} Z^{n-1}  \frac{X^n-Z^n}{X-Z}, $$
an inocuous modification of the Lucas sequence for $X$ and $Z$ (and independent of $Y$). 
\end{example}

\begin{example}
Similarly, taking powers of matrices $M \in \GL(2)$ leads to a determinantal divisibility sequence of the form 
$$ n \mapsto \frac{n^2}{\beta} \cdot \mathrm{det}(M)^{n-1} \cdot \left( \left(\frac{\alpha-\sqrt{\beta}}{2}\right)^n-\left(\frac{\alpha+\sqrt{\beta}}{2}\right)^n\right)^2 $$
with $\alpha=\mathrm{tr}(M)$ and $\beta=\mathrm{tr}^2(M)-4 \mathrm{det}(M)$. Here, when $\beta=0$ (i.e., the matrix has two identical eigenvalues), the formula should be understood in the limit as $\beta\rightarrow 0$, which gives $n^4(\alpha/2)^{4(n-1)}$.

It could be interesting to consider the determinantal divisibility sequence of more exotic linear algebraic groups. 
\end{example}

One might wonder whether for Lucas sequences, one can do with one dimension less, but this is not even true for Mersenne sequences and general sets of endomorphisms of the affine line, as a simple integration proves:

\begin{proposition}
A generalized Mersenne sequence $\{{x^n-1}\}_{n \geq 1}$ cannot occur as a matrix divisibility sequence associated to a set of endomorphisms of $\A^1$, i.e., in dimension $d=1$. 
\end{proposition}

\begin{proof}
If so, then there are polynomials $f_n$ such that $${x^n-1}=\frac{df_n}{dx}(x).$$ By integration, we find that $$f_n(x)=\frac{x^{n+1}}{n+1} - x + c_n $$ for some constants $c_n$, but then $n \mapsto f_n$ cannot be a representation, because it already fails to satisfy $f_{mn}=f_m \circ f_n$ (for example,  $\deg(f_m(f_n(x))) = (m+1)(n+1) \neq mn+1 = \deg(f_{mn}))$. \end{proof}

In connection with applications of divisibility sequences in logic, we record the following. Recall that a subset $X \subseteq \Z^d$ is called \emph{Diophantine} if there exists an algebraic variety $V$ defined over $\Z$ and a morphism $\pi \colon V \rightarrow \A^{d}$ defined over $\Z$, such that the image of the  set of integral points of $X$  is the given set: $\pi(V(\Z)) = X$.

\begin{proposition}
Suppose $\{M_n\}_{n \in \N}$ is a matrix divisibility sequence that arises as above from an affine algebraic group $\G/\Z$, evaluated at a point $P\in \G(\Z)$. Then $\{M_n\}_{n \in \N}$ is a Diophantine subset of $\Z^{d^2}$, and the associated determinant sequence $\{ \det(M_n) \}$ is a Diophantine subset of $\Z$.  \end{proposition}

\begin{proof}
By the David-Putnam-Robinson-Matijasevich theorem (see, e.g., \cite{DPRM}), Diophantine sets over $\Z$ are the same as recursively enumerable sets over $\Z$. We prove that the set $\{M_n\}$ is recursively enumerable. The formula that expresses $[n]x$ in algebraic terms, for a general point $x \in \G$, is computable in finite time on a Turing machine. The same holds for its Jacobian matrix. Hence also the values of the Jacobian matrices at $P$ are computable in finite time. Now the set $\{M_n\}$ can be enumerated by running through $n$. The same holds for the determinant sequence, since determinants are computable in finite time. 
\end{proof}

\begin{remark}
For a set of endomorphisms of a projective algebraic group (e.g., an abelian variety), one can use the general construction from Remark \ref{bundle}. One may also try to adapt the previous method from affine groups, by fixing an equation for addition in homogeneous coordinates and consider it on the affine cone over the group. (A particularly simple example of such a formula arises from the complete group law on the representation of an elliptic curve in Edwards form, cf.\ \cite{Edwards}, \cite{BL}) However, in general one will then only have a projective composition formula
$$ [\alpha \beta](P) = \lambda_{\alpha,\beta}(P) [\alpha]([\beta]P), $$
for some functions $\lambda_{\alpha, \beta}$ on $G$ --- from which the associated Jacobian matrix divisibility sequence will not in general be multiplicative, but rather satisfy 
$$ J_{\alpha \beta}(P) = \lambda_{\alpha,\beta}(P) J_{\alpha}([\beta]P) J_\beta(P) + (\nabla \lambda_{\alpha,\beta}(P))^\top \cdot [\beta]P. $$ 

If $P \in \G(S')$ is a point whose $\Gamma$-orbit stays within a fixed affine chart, then it is possible to extend the previous method. 

Another approach to general divisibility sequences, based on generalized GCD's, is due to Silverman \cite{Silverman}. For a further approach to (non-divisibility!) sequences in higher genus, see Cantor \cite{Cantor} (where the $r$-th division polynomial is zero at a point $P$ if and only if $rP$ is in the theta-divisor --- compare with \cite{Hone} for another interpretation of these sequences). 

In the next section, we will use a slightly different method for elliptic curves, based on the theory of division polynomials.
\end{remark}

\section{Matrix elliptic divisibility sequences: formal construction}

We will now show how elliptic divisibility sequences fit into the matrix divisibility picture, using division polynomials. 
Let $E$ denote a cubic curve with projective equation
\[
Y^2Z=X^3+AXZ^2+BZ^6.
\]
over the ring $S=\Z[A,B]$. The non-singular points of $E$ over any field containing $S$ form a group. Multiplication on the non-singular points of this cubic curve can be expressed using classical division polynomials
$$ n\cdot (x,y)= \left( \frac{\phi_n(x)}{\psi^2_n(x,y)}, \frac{\omega_n(x,y)}{\psi_n^3(x,y)} \right). $$
We refer to \cite{Cassels}, \cite{Ayad} and \cite{Silverman} (ex.\ III.3.7) for the definition of these polynomials. Here, $\phi_n$ and $$\widetilde{\psi}_n:=\psi_n^2$$ only depend on $x$. We now consider the following map of affine 2-space
$$ [n] \colon \A^2 \rightarrow \A^2 \colon (X,Z) \mapsto \left(Z^{n^2} \phi_n(\frac{X}{Z}), Z^{n^2} \widetilde{\psi}_n(\frac{X}{Z})\right). $$
The multiplicative property $(mn)P=m(nP)$ translates to $[mn]=[m] \circ [n]$ (compare \cite{Cassels}, formula (5)), so that $[\cdot]$ indeed defines a faithful representation of $\Gamma=\N$ as a group of endomorphisms of affine 2-space $\A^2$. Hence the associated sequence of Jacobian matrices is a matrix divisibility sequence, and its determinant is a divisibility sequence in the usual sense. We now establish a formula for these sequences in terms of known division polynomials. For this, we first compute some partial derivatives: 
\[
\frac{\partial X([n]\cdot(X,Z))}{\partial X}=Z^{n^2-1}\phi_n'(X/Z)
\]
and
\begin{eqnarray*}
\frac{\partial X([n]\cdot(X,Z))}{\partial Z}&=&-XZ^{n^2-2}\phi_n'(X/Z)+n^2Z^{n^2-1}\phi_n(X/Z) \\
&=&Z^{n^2-2}(n^2Z\phi_n(X/Z)-X\phi_n'(X/Z)).
\end{eqnarray*}
Also
\[
\frac{\partial Z([n]\cdot(X,Z))}{\partial X}=Z^{n^2-1}\widetilde{\psi}_n'(X/Z)
\]
and
\[
\frac{\partial Z([n]\cdot(X,Z))}{\partial Z}=Z^{n^2-2}(n^2Z\widetilde{\psi}_n(X/Z)-X\widetilde{\psi}_n'(X/Z)).
\]
We conclude:

\begin{proposition} \label{eds1}
The sequence
$$J_n(X,Z):= Z^{n^2-2} \left( \begin{array}{lr} Z\phi'_n(X/Z) & n^2 Z \phi_n(X/Z)-X \phi'_n(X/Z) \\ Z(\psi_n^2)'(X/Z) & n^2 \psi_n^2(X/Z)-X(\psi_n^2)'(X/Z)  \end{array} \right) $$
is a matrix divisibility sequence, which we call a \emph{matrix elliptic divisibility sequence}, with associated so-called 
\emph{determinant elliptic divisibility sequence}
$$ \det(J_n)(X,Z) = n^2 Z^{2(n^2-1)} W(\phi_n, \widetilde{\psi}_n)(X/Z), $$
where $W(f,g)=f'g-fg'$ is the Wronskian determinant of two functions $f,g$, and $\widetilde{\psi}_n:=\psi_n^2$. \qed
\end{proposition}

\begin{remark}
By Cassels' Theorem I in \cite{Cassels}, the polynomial derivatives $\phi_n'(x)$ and $\widetilde{\psi}_n'(x)$ have all their coefficients divisible by $n$; we conclude that the matrix $J_n(X,Z)$ is divisible by the diagonal matrix $\diag(n,n)$. 
\end{remark}

We can further simplify the Wronskian determinant in Proposition \ref{eds1}, as follows: by taking derivatives on both sides of $$x(n\cdot(x,y)) = \frac{\phi_n(x)}{\widetilde{\psi}_n(x)}$$ we find that 
$$ \frac{dx(n\cdot(x,y))}{dx} = \frac{W(\phi_n,\widetilde{\psi}_n)}{\widetilde{\psi}_n^2}. $$
To use $\wp$-functions, we switch to classical Weierstrass form, by writing $x=x_1/36$ and $y=y_1/432$, so that $(x_1,y_1)$ satisfies the Weierstrass equation in traditional form $y_1^2=4x_1^3-g_2 x - g_3$ for $g_2=-5184A$ and $g_3=-186624B$, and we can write $x_1=\wp(z), y_1=\wp'(z)$ for $\wp$ the Weierstrass $\wp$-function of the corresponding lattice. Then 
\[
\frac{dx(n\cdot(x,y))}{dx} =\frac{1}{36}\frac{d \wp(nz)}{dx}=\frac{n}{36}\wp'(nz)\frac{dz}{dx}=n\frac{\wp'(nz)}{\wp'(z)}, \]
which we further simplify to 
\[ n\frac{y([n](x,y))}{y}=\frac{2n}{\psi_2}y([n](x,y))=\frac{1}{\psi_n^4}
\left( \frac{n\psi_{2n}}{\psi_2} \right),
\]
so that we finally find
\begin{proposition} \label{eds2}
The determinant elliptic divisibility sequence from Proposition \ref{eds1} equals 
\[\det{J_n}(X,Z)=n^3 Z^{2(n^2-1)} \frac{\psi_{2n}}{\psi_2}(X/Z) = 2n^3 Z^{2(n^2-1)}  \frac{\psi_n \omega_n}{\psi_2} (X/Z). \qed
\]  
\end{proposition}
  
This result shows that every elliptic divisibility sequence occurs (up to passing to a field extension to divide a given point by $2$) as a determinant divisibility sequence. 

\begin{remark} 
We have already seen how Lucas sequences arise from the $2\times 2$ Borel group. Since all Lucas sequences also occur as elliptic divisibility sequence for singular cubics, we immediately find from the previous section that they, too, fit into this framework (\cite{Ward}, Thm. 22.1).
\end{remark}

\section{Matrix elliptic divisibility sequences: integral values and primitive divisors}

We now turn to the issue of actually substituting a rational point on the curve into these new sequences. 

\begin{proposition} \label{bf} Suppose that $P=(x,y)$ is a rational point of infinite order on an elliptic curve $E/\Q$ with chosen short Weierstrass equation with integral coefficients, and write $x=a/b^2$ in coprime integers $a,b$. 
The determinantal divisibility sequence $$\det J_n(a,b^2)$$ is integer valued, and has primitive prime divisors for $n$ sufficiently large. 
\end{proposition}

\begin{proof}
First of all, we quote a result of Ayad (\cite{Ayad}) to the effect that if we write $$nP=\left(\frac{A_n}{B_n^2}, y_n \right),$$ with $A_n, B_n$ coprime integers, then $$b^{2n^2} \psi^2_n(a/b^2) = B_n^2 Q_n,$$ where $Q_n$ is only divisible by primes $p$ for which $P$ is singular modulo $p$ on the given model (so in particular, $Q_n$ has only prime factors from the divisors of the discriminant $\Delta_E$ of the given curve). This means that $$\det J_n(a,b^2) = n^3 b^{4(n^2-1)} \frac{\psi_{2n}}{\psi_2}(a/{b^2}) = n^3 \frac{B_{2n}}{B_2}\cdot Q'_{n},$$ where $Q'_n$ has only prime divisors from $\Delta_E$.  

Now Silverman has proven the elliptic analogue of Zsigmondy's theorem \cite{SilPrim}, implying that $B_{2n}/B_2$ has a primitive prime divisor, say, $p$, for sufficiently large $n$ (since $P$ has infinite order in $E(\Q)$). We claim that $p$ is coprime to $n$ for $n$ sufficiently large. Indeed, suppose $p \mid n$. Since $p$ is prime and primitive, $P\ \mathrm{mod}\ p$ has order $2n$ in $E(\F_p)$, so that by the Hasse-Weil bound $$2n<p+1+2\sqrt{p}<n+1+2\sqrt{n},$$ leading to $n<6$. Hence for $n$ sufficiently large ($n>6$, $n$ large enough for Silverman's result to hold and for $2nP$ not to be an $S$-integer, where $S$ contains the primes dividing $\Delta_E$), $p$ is also primitive for $\det J_n(a,b)$.
\end{proof}

%
%
%
%
%
%
%

We finish this section by proving a matrix version of the existence of primitive divisors, based on the following general lemma:  

\begin{lemma}
Let $\{M_n\}_{n \in \N}$ denote a matrix divisibility sequence in integral matrices $M_n \in Mat_n(\Z)$. If the associated determinantal divisibility sequence $\{ \det(M_n) \}_{n \in \N}$ has primitive prime divisors, then the matrix divisibility sequence has primitive right divisor classes.
\end{lemma}

\begin{proof}
A nice way to organize the proof is by using the correspondence from Example \ref{B}, which implies that $M_n$ has a primitive right divisor if and only if $\Z^d/M_n^\top\Z^d$ has a subgroup that is not in the image of any of the natural reduction maps $\Z^d/M_m^\top \Z^d \rightarrow \Z^d/M_n^\top\Z^d$ for any $m \mid n$ with $m \neq n$. But since we assume that $\det(M_n)$ has a prime divisor $p$ that doesn't divide any $\det(M_m)$ for any $m \mid n$ with $m \neq n$, $p$ divides one of the elementary divisors of $M_n$, but none of those of such $M_m$. This implies that the subgroup $\Z^d/p \Z^d$ corresponding to $p$ has non-trivial reduction, so corresponds to a primitive right divisor class. 
\end{proof}

\begin{corollary}[Elliptic matrix Zsigmondy theorem] Suppose that $P=(x,y)$ is a rational point on an elliptic curve $E/\Q$ with chosen short Weierstrass equation with integral coefficients, and write $x=a/b^2$ in coprime integers $a,b$.
There exists an integer $N$ such that all the terms of a matrix elliptic divisibility sequence $\{J_n(a,b^2)\}_{n \in \N}$ with $n>N$ have primitive right matrix divisors. 
\end{corollary}

\begin{proof}
This follows from the previous lemma since the associated determinantal sequence (cf.\ Proposition \ref{eds1}) has primitive prime divisors by Proposition \ref{bf}. 
\end{proof}

\begin{remark}
One may also ask for `converse theorems' in the following style: if the height of the entries of the matrices $\{M_n\}$ has a specific growth behaviour in $n$, does it follow (at least generically) that its determinant sequence has a `related' growth behaviour?
\end{remark}

\begin{remark} Linear and elliptic divisibility sequences satisfy recurrence relations (provided their terms are chosen with the right sign), so we ask: 
Is there a choice of representatives for the divisor classes corresponding to a Lucas or elliptic \emph{matrix} divisibility sequence as in Proposition \ref{eds1}, such that these representative matrices themselves satisfy a polynomial recurrence relation (i.e., with coefficients that do not depend on the index of the term of the sequence)? We have checked by direct computation that it is \emph{not} the case that the ``Borel'' matrix sequence $J_n(X,Y,Z)$ from Example \ref{borel} satisfies a second order linear recurrence in matrices of the form $$J_n = A \cdot J_{n-1}+B \cdot J_{n-2}$$
for matrices $A=A(X,Y,Z)$ and $B=B(X,Y,Z)$ \emph{independent of $n$}. One might argue that in the non-commutative ring of matrices, a second order linear recurrence should be of the form 
$$ J_n = A \cdot J_{n-1} \cdot B+C \cdot J_{n-2}\cdot D $$
for matrices $A,B,C,D$ independent of $n$, but we did not investigate this possibility any further. 
\end{remark}

\bibliographystyle{amsplain}

\end{document}